\documentclass[11pt]{article}  

\usepackage{amsmath, amsthm, amsfonts, amssymb}
\usepackage[bookmarks]{hyperref}
\usepackage{indentfirst} 
\usepackage{marvosym}

\usepackage[a4paper]{geometry}

\newtheorem{lemma}{Lemma}[section]

\newtheorem{proposition}[lemma]{Proposition}

\renewenvironment{proof}[1][\proofname]{{\noindent\bf #1. }}{\qed}
\newtheorem{theoremletters}{Theorem}

\newtheorem{corollaryletters}[theoremletters]{Corollary}

\newcommand{\abs}[1]{\ensuremath{|#1|}}
\newcommand{\op}{\operatorname}
\newcommand{\ce}[2]{\pmb{\op{C}}_{#1}(#2)}

\newcommand{\ze}[1]{\pmb{\op{Z}}(#1)}

\newcommand{\rad}[2]{\pmb{\op{O}}_{#1}(#2)}
\newcommand{\syl}[2]{\op{Syl}_{#1}\left(#2\right)}
\newcommand{\hall}[2]{\op{Hall}_{#1}\left(#2\right)}

\begin{document}

\title{\bf Finite groups whose prime graph on class sizes is a block square}

\author{\sc V. Sotomayor\thanks{Departamento de Matemáticas y Estadística, Centro Universitario EDEM, Muelle de la Aduana s/n, La Marina de Valencia, 46024 Valencia, Spain.\newline
\Letter: \texttt{vsotomayor@edem.es} \newline \rule{6cm}{0.1mm}\newline
This research is supported by Proyecto PGC2018-096872-B-I00 from the Ministerio de Ciencia, Innovación y Universidades (Spain), and Proyecto AICO/2020/298 from the Generalitat Valenciana (Spain). \newline
}}

\date{}

\maketitle

\begin{abstract}
\noindent Let $G$ be a finite group, and let $\Delta(G)$ be the prime graph built on its set of conjugacy class sizes: this is the (simple undirected) graph whose vertices are the prime numbers dividing some conjugacy class size of $G$, and two distinct vertices $p, q$ are adjacent if and only if $pq$ divides some class size of $G$. In this paper, we characterise the structure of those groups $G$ whose prime graph $\Delta(G)$ is a \emph{block square}.

\medskip

\noindent \textbf{Keywords} Finite groups $\cdot$ Conjugacy classes $\cdot$ Prime graph

\smallskip

\noindent \textbf{2010 MSC} 20E45 
\end{abstract}

%%%%%%%%%%%%%%%%%%%%%%%%%%%%%%%%%%%%%%%%%%%%%%%%%%%%%%%%%%%%%%%%%%%%%%%%%%%%%%%%%%%%%%%%%%%%%%%%

\section{Introduction}

Throughout this paper, all groups considered are finite. Within finite group theory, the influence of the arithmetical properties of the conjugacy class sizes of a group on its algebraic structure is a research area that has attracted the interest of several authors over the last decades. The \emph{prime graph} built on the set of class sizes of a group $G$, which we denote by $\Delta(G)$, is a useful tool that is gaining an increasing interest for analysing the arithmetical properties of this set. This (simple undirected) graph has as vertex set $V(G)$ the prime divisors of the conjugacy class sizes of $G$, and its edge set $E(G)$ contains pairs $\{p,q\}\subseteq V(G)$ such that $pq$ divides some class size of $G$. In this framework, two relevant question that arise are: which graphs can occur as $\Delta(G)$ for some finite group $G$, and how is the structure of $G$ affected by the graph-theoretical properties of $\Delta(G)$?

Interestingly, non-adjacency between vertices of $\Delta(G)$ highly restricts the structure of $G$, which suggests that $\Delta(G)$ tends to have ``many'' edges. In fact, the extreme case when $\Delta(G)$ is disconnected happens if and only if $G$ is a $\mathcal{D}$\emph{-group}, that is, $G=AB$ where $A\unlhd G$ and $B$ are abelian subgroups of coprime orders, $\ze{G}\leqslant B$, and the factor group $G/\ze{G}$ is a Frobenius group with kernel $A\ze{G}/\ze{G}$ (see Theorem 4 of \cite{D}). In this situation, $G$ has three class sizes, which are $\{1, \abs{A}, \abs{B/\ze{G}}\}$. So the vertex sets of the (two) connected components of $\Delta(G)$ turn out to be the sets of prime divisors of the orders of $A$ and $B/\ze{G}$, respectively, and both sets are \emph{cliques} (i.e. they induce complete subgraphs) of $\Delta(G)$.

In \cite{few}, C. Casolo \emph{et al.} studied the structure of those finite groups $G$ such that $\Delta(G)$ has no complete vertices. Moreover, they characterised those groups whose prime graph on class sizes is non-complete and regular, and they are basically direct products of certain $\mathcal{D}$-groups. In particular, if $\Delta(G)$ is a square with $V(G)=\{p,q,r,s\}$ and $E(G)=\{\{p,r\},\{p,s\},\{q,r\},\{q,s\}\}$, then from their result it follows that (up to abelian direct factors) $G=A\times B$ where $A$ and $B$ are $\mathcal{D}$-groups of orders divisible by $\{p,q\}$ and $\{r,s\}$, respectively.

A natural way to generalise a square graph is to replace each vertex by a set of vertices. In this spirit, a graph is called a \emph{block square} if its vertex set can be written as a union of four disjoint, non-empty subsets $\pi_1, \pi_2, \pi_3, \pi_4$, where no prime in $\pi_1$ is adjacent to any prime in $\pi_4$ and no prime in $\pi_2$ is adjacent to any prime in $\pi_3$, and there exist vertices in both $\pi_1$ and $\pi_4$ that are adjacent to vertices in $\pi_2$ and in $\pi_3$. Certainly, any direct product $G=A\times B$ of two coprime $\mathcal{D}$-groups yields a block square $\Delta(G)$. So the question that naturally arises is whether there exist other types of groups whose prime graph on class sizes is a block square. The main result of this paper shows that in fact this is the unique way of obtaining groups with such class-size prime graph.

\begin{theoremletters}
\label{teoA}
Let $G$ be a finite group. Then $\Delta(G)$ is a block square if and only if, up to an abelian direct factor, $G=A\times B$ where $A$ and $B$ are $\mathcal{D}$-groups of coprime orders.
\end{theoremletters}

As a consequence, we have attained a characterisation of the block square graphs that can occur as $\Delta(G)$ for some finite group $G$.

\begin{corollaryletters}
\label{corB}
Let $\Delta$ be a block square graph. Then there exists a finite group $G$ such that $\Delta(G)=\Delta$ if and only if $\pi_i$ is a clique for each $1\leq i \leq 4$, and all the primes in $\pi_1 \cup \pi_4$ are adjacent to all the primes in $\pi_2\cup \pi_3$.
\end{corollaryletters}

Frequently, the results on the class-size context have a dual version in the context of degrees of irreducible characters. It is worth mentioning that M.L. Lewis and Q. Meng introduced in \cite{LM} the concept of block square graphs, and in that paper they carried out an analysis of block squares for the prime graph built on the character degrees of \emph{soluble} groups. Among other things, they proved an analogous version of Theorem \ref{teoA} for the character-degree prime graph in the particular case that the group possesses two normal non-abelian Sylow subgroups.

%%%%%%%%%%%%%%%%%%%%%%%%%%%%%%%%%%%%%%%%%%%%%%%%%%%%%%%%%%%%%%%%%%%%%%%%%%%%%%%%%%%%%%%%%%%%%%%%

\section{Preliminaries}

In the sequel, if $x$ is an element of a group $G$, then we denote by $x^G$ the conjugacy class of $x$ in $G$, and its size is $\abs{x^G}=\abs{G:\ce{G}{x}}$. For a positive integer $n$, we write $\pi(n)$ for the set of prime divisors of $n$, and in particular $\pi(G)$ is the set of prime divisors of $\abs{G}$. As usual, given a prime $p$, the set of all Sylow $p$-subgroups of $G$ is denoted by $\syl{p}{G}$, and $\hall{\pi}{G}$ is the set of all Hall $\pi$-subgroups of $G$ for a set of primes $\pi$. The remaining notation and terminology used is standard in the framework of finite group theory.

The following elementary properties will be used without further reference.

\begin{lemma}
Let $G$ be a group. Then the following conclusions hold.
\begin{itemize}
\item[\emph{(a)}] If either $x,y\in G$ have coprime orders and they commute, or $x\in M$ and $y \in N$ with $M$ and $N$ normal subgroups of $G$ such that $M\cap N=1$, then $\pi(\abs{x^G})\cup\pi(\abs{y^G})\subseteq\pi(\abs{(xy)^G})$.
\item[\emph{(b)}] A given prime $p$ does not lie in $V(G)$ if and only if $G$ has a central Sylow $p$-subgroup.
\end{itemize}
\end{lemma}

As it was mentioned in the Introduction, non-adjacency between vertices significantly constrains the structure of the group. The next result also illustrates this fact. It is Theorem C of \cite{incomplete}.

\begin{proposition}
\label{prop_edges}
Let $G$ be a group. If $\pi$ is a set of vertices which are all non-adjacent in $\Delta(G)$ to a vertex $p$, then $G$ is $\pi$-soluble with abelian Hall $\pi$-subgroups, and the vertices in $\pi$ are pairwise adjacent.
\end{proposition}

Observe that if $\Delta(G)$ is a block square, then it certainly has no complete vertices. Therefore the result below, which is Theorem C of \cite{few}, yields a reduction on the structure of such a group $G$.

\begin{proposition}
\label{reduction}
Let $G$ be group. Assume that no vertex of $\Delta(G)$ is complete. Then, up to an abelian direct factor, $G=KL$ with $K\unlhd G$ and $L$ abelian subgroups of coprime orders. Moreover, $K=G'$, $K\cap \ze{G}=1$, and both $\pi(K)$ and $\pi(L)$ are cliques of $\Delta(G)$.
\end{proposition}

We close this section with the next key fact, which is partially Proposition 3.1 of \cite{incomplete}.

\begin{proposition}
\label{key}
Let $G$ be a group, and $p,q$ non-adjacent vertices of $\Delta(G)$. Let $P\in\syl{p}{G}$ and $Q\in\syl{q}{G}$, and let $M$ be a non-trivial abelian normal subgroup of $G$ such that $\abs{M}$ is a power of a suitable prime $r$. Assume that $\ce{M}{P}=1$, and that $M$ has a complement in $G$. Then $\rad{q}{G}=Q\cap \ce{G}{M}\leqslant \ze{\ce{G}{M}}$.
\end{proposition}

%%%%%%%%%%%%%%%%%%%%%%%%%%%%%%%%%%%%%%%%%%%%%%%%%%%%%%%%%%%%%%%%%%%%%%%%%%%%%%%%%%%%%%%%%%%%%%%%

\section{Proof of main results}

\begin{proof}[Proof of Theorem \ref{teoA}]
First, recall that the class sizes of $G$ are the same that those of $G\times A$, where $A$ is an abelian group. Therefore, we may assume that $G$ has no abelian direct factors, and in particular $\pi(G)=V(G)$.

If $G=A\times B$ where $A$ and $B$ are $\mathcal{D}$-groups of coprime orders, then certainly $\Delta(G)$ is a block square, where $\pi_1$ and $\pi_4$ are respectively the set of prime divisors of the Frobenius kernel and complement of $A/\ze{A}$, and $\pi_2$ and $\pi_3$ are the set of prime divisors of the Frobenius kernel and complement of $B/\ze{B}$.

Therefore, the remainder of the proof is devoted to show that if $\Delta(G)$ is a block square, then $G=A\times B$ where $A$ and $B$ are $\mathcal{D}$-groups of coprime orders. In virtue of Proposition \ref{prop_edges}, we have that $\pi_i$ is a clique of $\Delta(G)$ and there exists an abelian Hall $\pi_i$-subgroup $H_i$ of $G$ for every $i\in\{1,2,3,4\}$. In particular, all the Sylow subgroups of $G$ are abelian.

Since there is no complete vertices in $\Delta(G)$, then by Proposition \ref{reduction} it follows that $G=KL$ where $K\unlhd G$ and $L$ are abelian subgroups of coprime orders, $K=G'$, $K\cap \ze{G}=1$, and both $\pi(K)$ and $\pi(L)$ induce complete subgraphs in $\Delta(G)$. In particular, $V(G)=\pi(K)\cup\pi(L)$.

Without loss of generality, we may assume that there exists a prime $p\in\pi_1\cap \pi(K)$, so $G$ has a normal (abelian) Sylow $p$-subgroup. As there is no edge in $\Delta(G)$ between $\pi_1$ and $\pi_4$, and $\pi(K)$ is a clique of $\Delta(G)$, then $\pi_4\subseteq \pi(L)$. Further, as $\pi(L)$ is also a clique of $\Delta(G)$, then necessarily it holds that $\pi_1\subseteq \pi(K)$, so $H_1=\rad{\pi_1}{G}\leqslant K$. Arguing analogously, we may suppose that $\pi_3\subseteq \pi(K)$ and $\pi_2\subseteq\pi(L)$. It follows that $K=H_1\times H_3$ and, up to conjugation, $L=H_2\times H_4$.

Next we proceed in three steps.

\noindent\textbf{\underline{Step 1:}} For each $s\in \pi_4$, there exists $p\in\pi_1$ such that $[P,S]\neq 1$, where $P\in\syl{p}{H_1}$ and $S\in\syl{s}{H_4}$. Besides, for each $r\in \pi_2$, there exists $q\in\pi_3$ such that $[Q,R]\neq 1$, where $Q\in\syl{q}{H_3}$ and $R\in\syl{r}{H_2}$.

In order to prove the first assertion, and arguing by contradiction, let us  assume that $[S, H_1]=1$. Since $L$ is abelian and $s\in V(G)$, then there necessarily exists $q\in \pi_3$ and $Q\in\syl{q}{H_3}$ such that $[Q,S]\neq 1$. As $Q\unlhd G$, there must exist some $y\in S$ with $q\in\pi(\abs{y^G})$.

Let $r\in\pi_2$ and $R\in\syl{r}{H_2}$. We claim that $[R, H_1]=1$. If not, then $R$ does not centralise some $P\in\syl{p}{H_1}$ for some $p\in \pi_1$. If $r$ does not divide $\abs{x^G}$ for each element $x\in P$, then $x\in \ce{G}{R^{g_x}}$ for some $g_x\in G$, and we may suppose that $g_x\in K$. But then $x=x^{g_x^{-1}}\in \ce{G}{R}$ because $x\in P\leqslant\ze{K}$. Since this is valid for all the elements $x\in P$, we get that $P\leqslant\ce{G}{R}$, a contradiction. Hence we may take an element $x\in P\leqslant H_1$ with $r\in\pi(\abs{x^G})$. As $[S,H_1]=1$, then $qr$ divides $\abs{(xy)^G}$, which is a contradiction because $q\in\pi_3$ and $r\in \pi_2$. 

Since the previous argument holds for each prime $r\in\pi_2$, we deduce that $H_2$ centralises $H_1$. But $H_1\leqslant \ze{K}$, so there exists $t\in \pi_4$ and $T\in\syl{t}{H_4}$ with $[H_1,T]\neq 1$ (in particular $t\neq s$). So we can take a suitable prime $p\in\pi_1$ such that $[P,T]\neq 1$ for $P\in\syl{p}{H_1}$. In particular, $p$ divides $\abs{w^G}$ for some $w\in T$.

Next we claim $\ce{Q}{T}=1$. Let us suppose that there exists a non-trivial element $x\in\ce{Q}{T}$. Certainly $K\leqslant\ce{G}{x}$, and since $K\cap \ze{G}=1$, then there exists a prime $u\in\pi_2\cup\pi_4$ such that $u$ divides $\abs{x^G}$. Recall that $[Q,S]\neq 1$ by the first paragraph, so we can pick an element $z\in Q$ with $s\in\pi(\abs{z^G})$. If $Q$ centralises $T$, then $\abs{(wz)^G}$ is divisible by both $p\in\pi_1$ and $s\in\pi_4$, a contradiction. Thus $Q$ does not centralise $T$, and therefore there exists an element $w_2\in T$ with $q\in\pi(\abs{w_2^G})$. Now we distinguish two cases: if $u\in\pi_2$, then the class size of $xw_2$ is divisible by $u$ and $q\in \pi_3$, a contradiction; if $u\in\pi_4$, then $\{u,p\}\subseteq \pi(\abs{(xw)^G})$, which is also a contradiction. Hence $\ce{Q}{T}=1$.

Recall that $Q$ is an abelian normal Sylow $q$-subgroup of $G$, so it is complemented in $G$. Since $\{p,t\}\notin E(G)$ for every $p\in\pi_1$, then Proposition \ref{key} leads to $P\leqslant \ze{\ce{G}{Q}}$ for $P\in\syl{p}{G}$, and this is valid for every prime $p\in \pi_1$. It follows $\ce{G}{Q}\leqslant\ce{G}{H_1}$.

Note that $\pi_1$ and $\pi_2$ are adjacent in $\Delta(G)$ by hypothesis, so there exist $v_1\in \pi_1$ and $v_2\in\pi_2$ such that $v_1v_2\in\pi(\abs{g^G})$ for some $g\in G$. We can decompose $g=g_kg_l$ in such way that $g_k\in K$, $g_l\in L$ up to conjugation, and $g_kg_l=g_lg_k$. In particular, since additionally $g_k^G$ and $g_l^G$ have coprime sizes, then $\abs{g^G}$ is the product of $\abs{g_k^G}$ and $\abs{ g_l^G}$. Therefore $v_1\in\pi(\abs{g_l^G})$ and $v_2\in\pi(\abs{g_k^G})$. As $\ce{G}{Q}\leqslant\ce{G}{H_1}$ by the previous paragraph, then certainly $g_l\notin \ce{G}{Q}$. This means that $q$ divides the class size of $g_l$, so $\abs{g^G}$ is divisible by both $v_2\in \pi_2$ and $q\in \pi_3$, a contradiction.

The first assertion of Step 1 is already proved. Observe that the second part analogously follows, since the roles of $\pi_1$ and $\pi_4$ are symmetric with respect to $\pi_3$ and $\pi_2$.

\noindent\textbf{\underline{Step 2:}} $H_1$ centralises $H_2$, and $H_3$ centralises $H_4$.

For proving that $[H_1, H_2]=1$, let us suppose that there exist $P\in\syl{p}{H_1}$ and $R\in\syl{r}{H_2}$ such that they do not commute, and we aim to reach a contradiction. Note that we can then take an element $z_1 \in R$ such that $p$ divides $\abs{z_1^G}$.

We claim $\ce{P}{R}=1$. Otherwise, there exists a non-trivial element $x\in \ce{P}{R}$, and since $K\cap \ze{G}=1$, then $\pi(\abs{x^G})$ contains a suitable prime $u\in \pi_2\cup \pi_4$. Using Step 1, there is a prime $q\in \pi_3$ and $Q\in\syl{q}{G}$ such that $Q$ does not centralise $R$. Hence $q$ divides the class size of certain element $z_2\in R$. It follows that the class sizes of $xz_1$ and $xz_2$ are divisible by $pu$ and $qu$, respectively. As $u\in \pi_2\cup \pi_4$, then this contradicts our assumptions. Therefore $\ce{P}{R}=1$.

Since $\{q,r\}\notin E(G)$ for every $q\in \pi_3$, then by Proposition \ref{key} we get that $Q\leqslant\ze{\ce{G}{P}}$, for $Q\in\syl{q}{G}$, and for all primes $q\in\pi_3$. It follows that $\ce{G}{P}\leqslant\ce{G}{H_3}$.

By assumptions, we can take an element $g\in G$ such that $v_3v_4\in \pi(\abs{g^G})$, where $v_3\in \pi_3$ and $v_4\in \pi_4$. This element can be written as a product of two suitable commutative elements $g_k\in K$ and $g_l\in L$. Since $\abs{g^G}$ is the product of the class sizes of $g_k$ and $g_l$, then necessarily we obtain that $v_3\in \pi(\abs{g_l})$ and $v_4\in\pi(\abs{g_k})$. Thus, in virtue of the above paragraph, $g_l$ cannot centralise $P$, and therefore $p\in\pi(\abs{g_l^G})$. It follows that $pv_3v_4$ divides $\abs{g^G}$, which is a contradiction because $p\in \pi_1$.

Hence $[H_1, H_2]=1$, and analogously it can be proved $[H_3,H_4]=1$. These two facts together with Step 1 yield $G=A\times B$, where $A:=H_1H_4$ with $H_1\unlhd A$ and $B:=H_3H_2$ with $H_3\unlhd B$. Note that $A$ and $B$ have coprime orders.

\noindent\textbf{\underline{Step 3:}} $A=H_1H_4$ and $B=H_3H_2$ are $\mathcal{D}$-groups.

We will show that $A=H_1H_4$ is a $\mathcal{D}$-group, and the same arguments are analogously valid for $B$. Note that $H_1$ and $H_4$ are abelian groups of coprime orders. Moreover, $\ze{A}\leqslant H_4$ since the Hall $\pi_1$-subgroup of $\ze{A}$ is contained in $K\cap \ze{G}=1$.

Let $Z:=\ze{A}$. We claim that $A/Z$ is a Frobenius group with Frobenius kernel $H_1Z/Z$. By coprime action, it is enough to prove that $\ce{H_4}{x}\leqslant Z$ for every non-trivial element $x\in H_1$. If this does not hold, then we can take an element $y\in \ce{H_4}{x}\smallsetminus Z$, so $\abs{y^A}$ is divisible by some prime $p\in \pi_1$. Since $x\notin Z$, then its class size in $A$ is divisible by certain prime $s\in \pi_4$. Therefore $ps\in\pi(\abs{(xy)^A})=\pi(\abs{(xy)^G})$, which contradicts our assumptions. 
\end{proof}

\begin{proof}[Proof of Corollary \ref{corB}]
In virtue on Theorem \ref{teoA}, the necessity of the condition is clear. Hence, let us suppose that $\Delta$ is a block square graph where $\pi_i$ is a clique for every $1\leq i \leq 4$, and all the primes in $\pi_1 \cup \pi_4$ are adjacent to all the primes in $\pi_2\cup \pi_3$. We aim to show that there exists a suitable group $G$ such that $\Delta(G)=\Delta$.

Let $m_i$ denote the size of each $\pi_i$, for $i\in\{1,2,3,4\}$. Let $n_1:=p_1p_2\cdots p_{m_1}$ where the $p_j$ are pairwise distinct prime numbers. Let $s_1,s_2,\ldots, s_{m_4}$ be distinct primes such that $n_4:=s_1s_2\cdots s_{m_4}$ is congruent to 1 modulo $n_1$; we point out that they exist by Dirichlet’s theorem on primes in an arithmetic progression. Let $K_1$ and $L_1$ be cyclic groups of orders $n_1$ and $n_4$, respectively. Consider the semidirect product $A=K_1\rtimes L_1$ with respect to a Frobenius action of $L_1$ on $K_1$. Certainly $A$ is a $\mathcal{D}$-group.

Now let $n_3:=q_1q_2\cdots q_{m_3}$ where $q_k\notin\pi(A)$ for each $k\in\{1,\ldots,m_3\}$ and they are pairwise distinct primes. Consider a set $\{r_1,r_2,\ldots, r_{m_2}\}$ of pairwise distinct primes such that none of them lies in $\pi(A)$ and $n_2:=r_1r_2\cdots r_{m_2}$ is congruent to 1 modulo $n_3$; again they exist by the aforementioned theorem due to Dirichlet. Let $K_2$ and $L_2$ be cyclic groups of orders $n_3$ and $n_2$, respectively. Consider the semidirect product $B=K_2\rtimes L_2$ with respect to a Frobenius action of $L_2$ on $K_2$, so $B$ is a $\mathcal{D}$-group. 

Certainly $(|A|, |B|)=1$. Let $G=A\times B$. Hence it easily follows that $\Delta(G)$ is a block square graph where all the vertices in $\pi(K_1)\cup\pi(L_1)$ are adjacent to all the vertices in $\pi(K_2)\cup \pi(L_2)$, and $\pi(K_i), \pi(L_i)$ are cliques for $1\leq i \leq 2$. Thus $\Delta(G)=\Delta$.
\end{proof}

%%%%%%%%%%%%%%%%%%%%%%%%%%%%%%%%%%%%%%%%%%%%%%%%%%%%%%%%%%%%%%%%%%%%%%%%%%%%%%%%%%%%%%%%%%%%%%%%

\noindent \textbf{Acknowledgements:} This research has been carried out during a stay of the author at the Dipartimento di Matematica e Informatica ``Ulisse Dini'' (DIMAI) of Università degli Studi di Firenze. He wishes to thank the members of the DIMAI for their hospitality, and the Centro Universitario EDEM for its support.

%%%%%%%%%%%%%%%%%%%%%%%%%%%%%%%%%%%%%%%%%%%%%%%%%%%%%%%%%%%%%%%%%%%%%%%%%%%%%%%%%%%%%%%%%%%%%%%%

\end{document}